\documentclass[11pt]{amsart}

\usepackage[a4paper,margin=2.5cm]{geometry}
\usepackage[T1]{fontenc}

\usepackage[dvipsnames]{xcolor}

\usepackage{mathtools}
\usepackage{amssymb}
\usepackage{mathrsfs}
\usepackage{stmaryrd}
\usepackage{newtxtext}
\usepackage{newtxmath}
\usepackage[bbgreekl]{mathbbol}
\usepackage{relsize}

\DeclareMathAlphabet{\mathcal}{OMS}{cmsy}{m}{n}

\DeclareSymbolFontAlphabet{\mathbb}{AMSb}
\DeclareSymbolFontAlphabet{\mathbbl}{bbold}

\usepackage{etoolbox}
\usepackage{float}
\usepackage[inline,shortlabels]{enumitem}
\usepackage{verbatim}
\usepackage[normalem]{ulem}
\usepackage{marginnote}
\usepackage[colorinlistoftodos,textsize=scriptsize]{todonotes}

\usepackage{thmtools}

\usepackage{tikz}
\usetikzlibrary{cd}

\usepackage[
  backend=biber,
  style=alphabetic,
  sorting=anyt,
  doi=false,
  url=false,
  isbn=false,
  giveninits=false,
  maxnames=6
]{biblatex}

\usepackage[pdfusetitle,linktocpage]{hyperref}

\newcommand{\myshade}{85}
\colorlet{mylinkcolor}{Red}
\colorlet{mycitecolor}{Cerulean}
\colorlet{myurlcolor}{Plum}

\hypersetup{
  linkcolor=mylinkcolor!\myshade!black,
  citecolor=mycitecolor!\myshade!black,
  urlcolor=myurlcolor!\myshade!black,
  colorlinks=true
}

\usepackage[capitalize,noabbrev]{cleveref}

\mathchardef\mhyphen="2D
\theoremstyle{plain}
\newtheorem{theorem}{Theorem}[section]

\newtheorem{proposition}[theorem]{Proposition}

\newtheorem{lemma}[theorem]{Lemma}

\newtheorem{corollary}[theorem]{Corollary}

\theoremstyle{definition}

\newtheorem{notation}[theorem]{Notation}

\newtheorem{remark}[theorem]{Remark}

\newcommand{\IA}{\mathbb{A}}

\newcommand{\IC}{\mathbb{C}}

\newcommand{\IF}{\mathbb{F}}

\newcommand{\IN}{\mathbb{N}}

\newcommand{\IQ}{\mathbb{Q}}
\newcommand{\IR}{\mathbb{R}}

\newcommand{\IZ}{\mathbb{Z}}

\newcommand{\sB}{\mathcal{B}}
\newcommand{\sC}{\mathcal{C}}

\newcommand{\sL}{\mathcal{L}}

\newcommand{\sO}{\mathcal{O}}

\newcommand{\sU}{\mathcal{U}}

\newcommand{\sX}{\mathcal{X}}

\newcommand{\End}{\mathrm{End}}

\newcommand{\Hom}{\mathrm{Hom}}

\renewcommand{\deg}{\mathrm{deg}}

\newcommand{\Gal}{\mathrm{Gal}}

\renewcommand{\H}{\mathrm{H}}

\newcommand{\tensor}{\otimes}

\newcommand{\suchthat}{\;\ifnum\currentgrouptype=16 \middle\fi\vert\;}

\newcommand{\restr}[2]{{\left.\kern-\nulldelimiterspace#1\vphantom{\big|}\right|_{#2}}}

\newcommand{\into}{\hookrightarrow}

\def\d/{/\mspace{-6.0mu}/}

\newcommand{\wt}{\widetilde}

\newcommand{\crys}{\mathrm{crys}}

\DeclareTextCommand{\polhk}{T1}{\k}

\makeatletter
\def\paragraph{\@startsection{paragraph}{4}%
  \z@\z@{-\fontdimen2\font}%
  {\normalfont\bfseries}}

\def\subparagraph{\@startsection{subparagraph}{5}%
  \z@\z@{-\fontdimen2\font}%
  {\normalfont\bfseries}}
\makeatother

\newcommand{\Zar}{\mathrm{zar}}

\title{Faltings' Isogeny Theorem via Equidistribution}
\author{Junyi Xie}

\address{Beijing International Center for Mathematical Research, Peking University, Beijing 100871, China}

\email{xiejunyi@bicmr.pku.edu.cn}

\author{Ziquan Yang}

\address{Department of Mathematics \& Institute of Mathematical Sciences, Chinese University of Hong Kong, Shatin, N.T., Hong Kong SAR, China}

\email{zqyang@math.cuhk.edu.hk}

\begin{document}

\begin{abstract}
     We give a new proof of Faltings' isogeny theorem. More precisely, we show that Yuan's non-archimedean equidistribution theorem can be used to ``pump'' homomorphisms and semisimplicity from finite fields to number fields, thereby reducing Faltings' theorem directly to Tate's theorem. 
    In the appendix joint with Yap, we explain how an ultraproduct variant of our method gives another proof of the Shafarevich conjecture, and hence a new route to the Mordell conjecture via Parshin's trick. 
\looseness=-1
\end{abstract}

\maketitle

\vspace{-1em}
\section{Introduction}

In \cite{SUZ97}, Szpiro, Ullmo and Zhang proved that for a generic sequence of torsion points $\{ x_n \}$ on an abelian variety $A$ over a number field $K$, the Galois orbits of $x_n$ become equidistributed as $n \to \infty$. Let us draw a picture to illustrate their theorem:

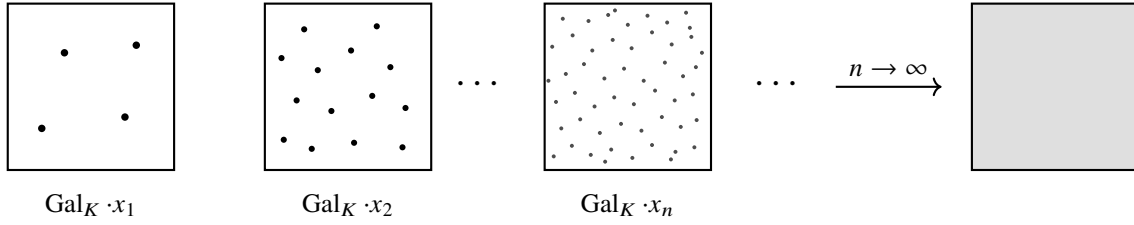
\begin{center}
    \begin{tikzpicture}[
    scale=1,
    every node/.style={font=\small}
]

\def\s{2.2}      
\def\gapa{1.2}   
\def\gapb{1.5}   
\def\gapc{2.2}   

\begin{scope}[shift={(0,0)}]
    \draw[thick] (0,0) rectangle (\s,\s);

    \foreach \x/\y in {
        0.45/0.55,
        1.55/0.70,
        0.75/1.55,
        1.70/1.65
    }{
        \fill (\x,\y) circle (1.4pt);
    }

    \node[below] at (\s/2,-0.18) {$\operatorname{Gal}_K\cdot x_1$};
\end{scope}

\begin{scope}[shift={(\s+\gapa,0)}]
    \draw[thick] (0,0) rectangle (\s,\s);

    \foreach \x/\y in {
        0.25/0.40,
        0.62/0.28,
        1.18/0.36,
        1.82/0.30,
        0.42/0.92,
        0.88/0.78,
        1.42/0.98,
        1.86/0.82,
        0.22/1.48,
        0.70/1.32,
        1.14/1.58,
        1.66/1.36,
        0.52/1.86,
        1.48/1.90
    }{
        \fill (\x,\y) circle (1.15pt);
    }

    \node[below] at (\s/2,-0.18) {$\operatorname{Gal}_K\cdot x_2$};
\end{scope}

\node at ({2*\s+\gapa+0.65}, {0.5*\s}) {\Large $\cdots$};

\begin{scope}[shift={(2*\s+\gapa+\gapb,0)}]
    \draw[thick] (0,0) rectangle (\s,\s);

    \foreach \x/\y in {
        0.12/0.18, 0.36/0.32, 0.61/0.20, 0.88/0.27, 1.16/0.16, 1.42/0.33, 1.73/0.24, 1.98/0.30,
        0.22/0.55, 0.47/0.67, 0.73/0.49, 1.01/0.61, 1.28/0.52, 1.55/0.70, 1.80/0.57, 2.01/0.66,
        0.15/0.90, 0.40/1.02, 0.66/0.84, 0.92/0.96, 1.18/0.87, 1.46/1.05, 1.72/0.92, 1.96/1.00,
        0.28/1.27, 0.52/1.40, 0.79/1.21, 1.05/1.34, 1.31/1.18, 1.59/1.42, 1.83/1.24, 2.00/1.36,
        0.10/1.63, 0.37/1.78, 0.63/1.58, 0.90/1.72, 1.15/1.61, 1.42/1.82, 1.68/1.66, 1.94/1.76,
        0.24/2.00, 0.58/1.97, 0.84/2.05, 1.10/1.92, 1.36/2.04, 1.62/1.95, 1.87/2.08,
        0.80/0.11, 1.67/0.14, 0.05/1.18, 2.08/1.55, 1.92/1.88, 0.93/2.11
    }{
        \fill[black!70] (\x,\y) circle (0.75pt);
    }

    \node[below] at (\s/2,-0.18) {$\operatorname{Gal}_K\cdot x_n$};
\end{scope}

\node at ({3*\s+\gapa+\gapb+0.9}, {0.5*\s}) {\Large $\cdots$};

\draw[->, thick]
    ({3*\s+\gapa+\gapb+\gapc-0.55}, {0.5*\s})
    --
    ({3*\s+\gapa+\gapb+\gapc+0.85}, {0.5*\s})
    node[midway, above] {$n\to\infty$};

\begin{scope}[shift={(3*\s+\gapa+\gapb+\gapc+1.25,0)}]
    \fill[gray!25] (0,0) rectangle (\s,\s);
    \draw[thick] (0,0) rectangle (\s,\s);

    \node[below] at (\s/2,-0.18) {};
\end{scope}

\end{tikzpicture}
\end{center}

In the main text, we demonstrate how to go from the picture above (or rather its non-archimedean version due to Yuan \cite{Yuan08}) to Faltings' isogeny theorem \cite{Fal83} for abelian varieties over number fields via a quick and elementary argument, assuming the corresponding theorem over finite fields due to Tate \cite{Tat66}. In the appendix, it is explained how to go from this picture all the way to the Mordell conjecture via an ultraproduct variant of the method. It is interesting to note that, while equidistribution a priori governs small points, and the Mordell conjecture concerns rational points with no a priori height bound, there is nevertheless a highway from equidistribution to Mordell.

Let us recall Faltings' isogeny theorem below.

\begin{theorem}
\emph{(\cite[Satz 3, 4]{Fal83})}
\label{thm: main}
    Let $A$ be an abelian variety over a number field $K$, and $\ell$ be any prime. 
    \begin{enumerate}[label=\upshape{(\alph*)}]
        \item The natural map $\End_K(A) \tensor \IZ_\ell \to \End_{\Gal_K} (T_\ell A)$ is an isomorphism.
        \item The action of $\Gal_K$ on $V_\ell A$ is semisimple. 
    \end{enumerate}
\end{theorem}

By applying (a) to a product, it is easy to see that (a) is equivalent to saying that when $A, B$ are abelian varieties over $K$, then $\Hom_K(A, B) \tensor \IZ_\ell \to \Hom_{\Gal_K} (T_\ell A, T_\ell B)$ is an isomorphism. We shall therefore refer to part (a) as the ``Hom theorem'', and we refer to part (b) as the ``semisimplicity theorem''. A consequence of \cref{thm: main} is the following: 

\begin{corollary}
    \emph{(\cite[Korollar~2]{Fal83})}\label{cor: isogenous}
    Two abelian varieties over a number field $K$ are isogenous over $K$ if and only if they have the same L-factors at almost all places of $K$. 
\end{corollary}

One may choose to refer to either \cref{thm: main} or \cref{cor: isogenous} as the ``Faltings isogeny theorem'', and the latter is a formal consequence of the former via Chebotarev density and Brauer--Nesbitt. In this note, we focus on \cref{thm: main}, which is a special case of the Tate conjecture.


\medskip

\paragraph{Notation and conventions}
For an abelian variety $A$ over a field $\kappa$ and a prime $\ell \neq \mathrm{char}\,\kappa$, we write $T_\ell A := \varprojlim_n A(\bar{\kappa})[\ell^n]$ for the Tate module of $A$, and $V_\ell A := T_\ell A \tensor_{\IZ_\ell} \IQ_\ell$. We shall write $A(\bar{\kappa})[\ell^\infty]$ simply as $A[\ell^\infty]$. There is a canonical identification $A[\ell^\infty] = V_\ell A / T_\ell A$. Thus, if $B$ is another abelian variety over $\kappa$, any $\IZ_\ell$-linear homomorphism $\psi : T_\ell A \to T_\ell B$ extends uniquely to a $\IQ_\ell$-linear homomorphism $V_\ell A \to V_\ell B$, and hence induces a homomorphism $\psi : A[\ell^\infty] \to B[\ell^\infty]$. We shall use the same notation $\psi$ for all three maps without further comment.

An observation we shall repeatedly use throughout the paper is the following. 
\begin{lemma}
\label{lem: ab sub}
    In the notation above, let $H$ be a $\Gal_\kappa$-invariant and $\ell$-divisible subgroup of $A[\ell^\infty]$. Then the Zariski closure $\overline{H}^\Zar$ is an abelian subvariety. 
\end{lemma}
\begin{proof}
    Clearly $\overline{H}^{\Zar}$ is closed under the group structure. Connectedness follows from the fact that the component group is both $\ell$-divisible and $\ell$-primary, and hence is trivial. 
\end{proof}

\section{Conceptual remarks}

Since the point of this paper is not the truth of the main theorem, which Faltings already proved in the 1980s, we take some time to explain in detail what is conceptually new in our approach. Depending on the reader's taste, this section may be read either before or after the proofs. 

The most naive idea of proving the algebraicity part of the Tate conjecture involves two steps: 

\begin{enumerate}[label=(\Roman*)]
    \item Find a way of constructing lots of algebraic cycles. 
    \item Argue that every Tate class is given by an $\ell$-adic linear combination of algebraic cycles from (I). 
\end{enumerate}
We will show how to quite literally carry out (I) and (II) for homomorphisms between abelian varieties over number fields. We remark that Faltings' strategy is more akin to developing, over number fields, a parallel story to Tate's (further developed by Zarhin), whereas our strategy uses equidistribution to directly \textbf{``pump''} homomorphisms and semisimplicity from a finite field to a number field. This is the most important difference in spirit between our approach and Faltings'. The pumping mechanism is very transparent for homomorphisms, and is more delicate for semisimplicity (see \cref{rmk: pump semisimplicity}). \\

\paragraph{Construction of cycles} First, how does equidistribution help with ``constructing'' algebraic cycles? The mechanism is quite simple. Suppose that you have infinitely many points on an irreducible variety. Then you take their Zariski closure. If the Zariski closure is not the whole variety, then it gives you a nontrivial algebraic cycle.

Equidistribution often gives a statement of the following type: if a sequence of small points is Galois-stable, then it should be equidistributed on its Zariski closure. This puts a strong constraint on what the Zariski closure can be. Over $\IC$, under suitable conditions, this can force the Zariski closure into the Euclidean closure (i.e., the closure in the classical analytic topology); over a $p$-adic field, this can mean forcing the Zariski closure into a $p$-adic tube, as we shall see. \\

\paragraph{Non-archimedean rationality} The Tate conjecture is usually viewed as an analogue of the Hodge conjecture. Both conjectures are hard because our methods for producing algebraic cycles are quite limited, but the Tate conjecture has an additional difficulty: unlike the Hodge conjecture, one cannot take a Tate class and try to show directly that it is given by an actual algebraic cycle\footnote{By which we mean an element of a Chow group with rational coefficients.}, because by purely linear-algebraic means it is not possible to distinguish between classes directly given by algebraic cycles and those which are merely $\ell$-adic, and possibly irrational, linear combinations of algebraic classes.

This contrast is especially sharp in a situation where the corresponding Hodge-theoretic statement is almost trivial. For abelian varieties $A,B$ over $\IC$, the fact that
$$ \Hom(A,B)=\Hom_{\mathrm{Hdg}}(\H^1(B,\IZ),\H^1(A,\IZ)) $$
is immediate from a complex analytic point of view, but the corresponding Tate and Faltings theorems remain much more delicate.

Therefore, one needs a notion of ``rationality'' for a Tate class before proving that it literally comes from an algebraic cycle. For a variety over a number field $K$, every archimedean place $\sigma:K\into\IC$ gives a notion of rationality: we can say that a Tate class is rational if it comes from the Betti cohomology, with $\IQ$-coefficients, of the resulting complex manifold. This is the very idea behind the notion of absolute Hodge cycles. But one can also consider a notion of rationality given by a non-archimedean place. Namely, if $v$ is a place of good reduction, we can say that a Tate class is rational if, on the special fiber, it is given by an actual algebraic cycle.

With this notion of rationality, it is still hard to know whether a given Tate class is rational. However, for the Hom theorem, we manage to prove that the space of Tate classes as a whole is rational (i.e., admits a rational basis). This is good enough for the Tate conjecture. We achieve this again using equidistribution, via a ``diagonal trick''.  \\

\paragraph{Algebraicity versus semisimplicity} One reason that cycle conjectures are notoriously difficult is that their various components are often intertwined. One frequently encounters a situation in which proving A requires B, while the most natural proof of B in turn requires A. This is particularly true of the algebraicity and semisimplicity statements in the Tate conjecture. For example, Tate used semisimplicity of Frobenius (known to Weil) to prove the Hom theorem. On the other hand, Moonen showed that over a number field, the algebraicity part of the (strong) Tate conjecture implies the semisimplicity part (\cite{Moo19}).

There are two ways to deal with this circularity: either one proves the two statements simultaneously --- as in Faltings' original proof --- or one finds a way to break the loop. Our approach does the latter. More precisely, we first prove the Hom theorem and then deduce the semisimplicity theorem by a bootstrap argument. This is a bootstrap in two senses: the Hom theorem is an input for semisimplicity, and the main argument for semisimplicity is a noncommutative bootstrap of the diagonal trick. Note that this goes in the opposite direction to what Tate did. \\

\paragraph{Finiteness input} It is well known that the Tate conjecture should reflect the finiteness of certain objects in arithmetic geometry (cf. Totaro's beautiful survey \cite{Tot17}). This is clearly manifested in Tate's and Faltings' original approaches to their theorems. In particular, Faltings proved his isogeny theorem first by proving a finiteness property of polarized abelian varieties of bounded height. 

By contrast, our argument does not go through such finiteness statements over number fields. Of course, since we invoke Tate's theorem, finiteness still enters the argument. But the finiteness used by Tate is of a very elementary kind. Namely, a moduli space of finite type has finitely many points over a finite field. Other than that, we do not use finiteness of certain arithmetic-geometric objects specific to number fields. Nonetheless, the equidistribution we use, just like finiteness statements, is a reflection of the ``smallness'' of number fields --- they need to have sufficiently rich Galois groups.

\section{An archimedean toy example}

In this section, we give an archimedean toy example of how to use equidistribution to ``construct'' algebraic cycles. 

\begin{theorem}
\label{thm: Junyi over C}
    Suppose that $A, B$ are abelian varieties over a number field $K \subseteq \IC$. 
    Let $\psi : \H^1(B(\IC), \IZ) \to \H^1(A(\IC), \IZ)$ be a morphism of lattices. Then $\psi$ is induced by an algebraic morphism $f : A \to B$ if and only if $\psi \tensor \IZ_\ell$ is $\Gal_K$-equivariant. 
\end{theorem}

\begin{proof}
    Suppose that $A = \IC^a / \Lambda_A$ and $B = \IC^b / \Lambda_B$. By dualizing $\psi$, we can reinterpret $\psi$ as a map between lattices $\Lambda_A \to \Lambda_B$. 
    Then $ \psi \tensor_\IZ \IR$ gives rise to a map $A(\IC) \to B(\IC)$ as real tori. Let us denote it by $f_\IC$. 
    Choose an embedding $\sigma : \bar{K} \into \IC$ and omit it from the notation. 
    Choose some prime $\ell$ and consider $\psi : A[\ell^\infty] \to B[\ell^\infty]$. 
    Consider the graph 
    $$ \Gamma_{\psi} = \{ (x, \psi(x)) : x \in A[\ell^\infty] \} \subseteq A(\bar{K}) \times B(\bar{K}).  $$
    Now we take the Zariski closure $C = \overline{\Gamma}^\Zar_{\psi}$ in $A \times B$, which is an abelian subvariety by \cref{lem: ab sub}.  
    
    It suffices to show $C(\IC) = \Gamma_{f_\IC}$. Note also that $\Gamma_{f_\IC}$ is the closure of $\Gamma_\psi$ in the classical (Euclidean) topology, so $\Gamma_{f_\IC} \subseteq C(\IC)$. To see the reverse inclusion, we shall apply Szpiro-Ullmo-Zhang's equidistribution theorem. 
    Now let us choose a sequence $\{ y_n = (x_n, \psi(x_n)) \} \subseteq \Gamma_\psi$ such that no subsequence is contained in a torsion translate of proper abelian subvarieties of $C$. By \cite[Thm.~1.1]{SUZ97}, the Galois orbits of $y_n$'s are equidistributed on $C(\IC)$. On the other hand, as $\Gamma_{\psi}$ is $\Gal_K$-stable, the Galois orbits of $y_n$'s are all contained in $\Gamma_{f_\IC}$. This implies that $C(\IC) \subseteq \Gamma_{f_\IC}$. 
\end{proof}

\section{A non-archimedean analogue}
For readers' convenience we recall Yuan's generic equidistribution of small points \cite[Thm.~3.1]{Yuan08}. We will apply it to non-archimedean places. 

\begin{theorem}
    \emph{(Yuan)} Let $X$ be a projective variety of dimension $g$ over a number field $K$, and let $\overline{L} = (L, \| \cdot \|)$ be an adelic metrized line bundle such that $L$ is ample and the metric is semipositive. Let $\{ x_n \}$ be an infinite sequence of $\bar{K}$-points which is generic and small. 
    
    Then for any place $v$ of $K$, the Galois orbits of $\{ x_n \}$ are equidistributed in the analytic space $X_{\IC_v}^{\mathrm{an}}$ with respect to the measure $d \mu_v = c_1(\overline{L})^g_v / \deg_L(X)$. 
\end{theorem}
Recall that in \textit{loc. cit.}, a sequence is called ``generic'' if it has no subsequence contained in a proper closed subvariety, and "small" if $h_{\overline{L}}(x_n) \to h_{\overline{L}}(X)$.

\begin{proposition}
\label{prop: tube absorbance}
Let $X$ be an abelian variety over $K$, and let $v$ be a finite place such that $X$ has good reduction $\sX$ over $\sO_v$. Let $k = k(v)$ be the residue field and assume $\mathrm{char\,} k \neq \ell$. Let $X_0 = \sX_{k}$ be its special fiber and $Z_0 \subseteq X_0$ be a closed subvariety. Let $H \subseteq X[\ell^\infty]$ be an $\ell$-divisible, $\Gal_K$-stable abelian subgroup. Then if $H$ is contained in the tube $]Z_0[$ defined by $Z_0$, so is $\overline{H}^\Zar$. 
\end{proposition}
\begin{proof}
    Again write $C$ for the Zariski closure $\overline{H}^\Zar$. Then $C$ is an abelian subvariety. As $H$ is $\ell$-divisible, $C$ is connected. 
    One easily deduces from the N\'eron-Ogg-Shafarevich criterion that since $X$ has good reduction at $v$, so does $C$. Let $\sC$ be its extension over $\sO_v$ and write $C_0$ for its special fiber. Assume that $H \subseteq ]Z_0[$. Our goal is to show that $C_0 \to X_0$ factors through $Z_0$.  

    Suppose that the image of $C_0$ is not contained in $Z_0$. Let $L$ be a symmetric ample line bundle on $X$ such that its extension $\sL$ over $\sX$ remains ample. Endow $L$ with the canonical metric $\| \cdot \|$. Then at $v$, $\| \cdot \|_v$ is the same as the model metric given by the model $(\sX, \sL)$ (cf. \cite[\S2.3]{Zha95}). Up to replacing $\sL$ by a sufficiently high power, one can always find a global section $s \in \H^0(\sL)$ which vanishes on $Z_0$ but does not vanish on $C_0$. Consider the function $f : X^{\mathrm{an}}_{\IC_v} \to \IR$ defined by 
    $$ f(x) = \min \{ - \log \| s(x) \|_v, 1 \}. $$
    Here we take the normalization that $|\pi|_v = q_v^{-1}$, where $\pi$ is a uniformizer of $\sO_v$ and $q_v$ is the order of its residue field.

    Let $\{ x_n \} \subseteq H$ be a generic sequence of points on $C$ and consider the equation 
    \begin{equation}
    \label{eqn: Yuan for C}
         \lim_{n \to \infty} \left( \frac{1}{\# O(x_n)} \sum_{x\in O(x_n)} f(x) \right) = f(\xi_C)
    \end{equation}
    given by \cite[Thm.~3.1]{Yuan08}, where $\xi_C$ is the Shilov point of $C$, and $O(x_n) = \Gal_K \cdot x_n$ is the Galois orbit.  
    The fact that $f(\xi_C)$ is the integral of $f$ with respect to $\mu_v$ follows from \cite{CL06}. 
    As $s$ does not vanish on $C_0$, $f(\xi_C) = 0$. On the other hand, the $\ell$-primary torsion points of $X$ are defined over unramified extensions at $v$. Since $s$ vanishes on $Z_0$ and every Galois conjugate of $x_n$ reduces to $Z_0$, we have
    $\|s(x) \|_v \leq |\pi|_v = q_v^{-1} $ for every $x\in O(x_n)$. Therefore $f(x)\geq \min\{\log q_v, 1 \}>0$, and consequently the left hand side of (\ref{eqn: Yuan for C}) is strictly positive, a contradiction. 
\end{proof}

We can use the above to prove a non-archimedean analogue of \cref{thm: Junyi over C}.

\begin{theorem}
\label{thm: Junyi in char p}
    Let $A, B$ be abelian varieties over a number field $K$. Let $v$ be a place where both have good reduction and denote their special fibers by $A_0, B_0$. Let $k$ be the residue field and $p := \mathrm{char\,} k$. Assume that $e(\sO_v / \IZ_p) < p - 1$. 
    
    Let $\ell \neq p$ be another prime. Under the identifications $T_\ell A = T_\ell A_0$ and $T_\ell B = T_\ell B_0$, if $f_0 : A_0 \to B_0$ is a homomorphism on the special fiber such that $T_\ell(f_0) : T_\ell A \to T_\ell B$ is $\Gal_K$-equivariant, then $f_0$ lifts to a homomorphism $f : A \to B$. 
\end{theorem}
\begin{proof}
    On $A \times B$, consider $H = \{ (x, f_0(x)) : x \in A[\ell^\infty] \} \subseteq (A \times B)(\bar{K})$. 
    As before $C := \overline{H}^{\Zar}$ is an abelian subvariety with good reduction. Let $C_0$ be its reduction. By \cref{prop: tube absorbance}, the reduction of the inclusion $C \into A \times B$ is a morphism $C_0 \to A_0 \times B_0$ which factors through $\Gamma_{f_0}$. Since we assumed that $e(\sO_v/ \IZ_p) < p - 1$, by \cite[\S7.5~Thm.~4]{BLR90}, $C_0 \to A_0 \times B_0$ remains an inclusion, so that for dimension reasons $C_0 \to \Gamma_{f_0}$ is actually an isomorphism. Therefore, $C$ is a lift of the graph $\Gamma_{f_0}$ and hence defines a homomorphism $f : A \to B$ which lifts $f_0$. 
\end{proof}

\begin{remark}
    If we do not impose $e(\sO_v/ \IZ_p) < p - 1$, then a slightly weaker condition holds: the morphism $C_0 \to A_0 \times B_0$, and hence $C_0 \to \Gamma_{f_0}$, might have a finite $p$-primary kernel (cf. \cite[\S7.5, Ex.8]{BLR90}). 
    One might compare \cref{thm: Junyi in char p} with the following consequence of Grothendieck-Messing + Serre-Tate (cf. \cite[Ch.~V, \S1, Thms.~1.6 and~1.10]{Mes72}, or \cite[Thm.~3.15 and~(3.8)]{BO83}). Suppose now $K$ denotes a $p$-adic field of ramification degree $< p - 1$ and $A, B$ be abelian schemes over $\sO_K$. Then $f_0 : A_0 \to B_0$ lifts to $A \to B$ if and only if $\H^1_\crys(f_0)$ respects the Hodge filtrations under the crystalline-de Rham comparison. One may compare \cite[\S7.5, Ex.~8]{BLR90} with \cite[Ex.~3.18]{BO83} for what happens when $e = p - 1$. 
\end{remark}

\section{The Hom Theorem}

Let us begin with a few simple linear algebra tricks.

\begin{lemma}
\label{lem: find basis}
    Let $V$ be a finite dimensional vector space over $\IQ$, and $W_\ell$ be a subspace of $V_\ell := V \tensor \IQ_\ell$. If $W_\ell$ is not rational, i.e., $W_\ell \neq (W_\ell \cap V) \tensor_\IQ \IQ_\ell$, then there exists a basis $\{ w_1, \cdots, w_r, e_1, \cdots, e_s \}$ for $V$ such that $\{ w_i \}$ is a basis for $W_\ell \cap V$, and $W_\ell$ contains an element $\epsilon$ of the form
    \begin{equation}
    \label{eqn: decompose e}
        \epsilon =  \sum_{i = 1}^m c_i e_i
    \end{equation}
    where $c_1, \cdots, c_m \in \IQ^\times_\ell$ are linearly independent over $\IQ$ for some $m \ge 2$.
\end{lemma}
\begin{proof}
    Choose a complement $E$ of $W_\ell \cap V$ in $V$. 
    Fix a basis $\{ w_i \}_{i = 1}^r$ for $W_\ell \cap V$, and let $\sB$ be the set of all possible bases for $E$. As $W_\ell$ is not rational, it must contain some $\epsilon \in E_\ell := E \tensor \IQ_\ell$ such that $\IQ_\ell \epsilon \cap E = 0$. Fix this element and define $\alpha : \sB \to \IN$ which sends every $\{ e_j \}$ to the number of non-zero coordinates of $\epsilon$ with respect to basis $\{ e_j \}$. 
    It is not hard to see that $\alpha(\{ e_j \}) \ge 2$ for any $\{ e_j \} \in \sB$. Now let us choose some $\{ e_j \}$ that minimizes $\alpha$. Up to relabelling, we may assume $\epsilon$ decomposes as in (\ref{eqn: decompose e}). 

    We claim that $\{ c_1, \cdots, c_m \}$ is linearly independent over $\IQ$. Otherwise, up to relabelling $e_j$'s, there exist $q_j$'s in $\IQ$ such that
    $$ c_1 = \sum_{j = 2}^m q_j c_j \Rightarrow \epsilon = (\sum_{j = 2}^m q_j c_j) e_1 + (\sum_{j = 2}^m c_j e_j) = \sum_{j = 2}^m c_j (q_j e_1 + e_j). $$
    This contradicts the minimality of $\alpha(\{ e_j \})$, because $\{ e_1, e_{m + 1}, \cdots, e_s \} \cup \{ q_j e_1 + e_j \}_{2 \le j \le m} \in \sB$. 
\end{proof}

\begin{lemma}
\label{lem: irrationality}
Let $A$ be a simple abelian variety over a field $\kappa$, let $\ell \neq \mathrm{char\,}\kappa$, and set $D=\End_\kappa^0(A)$. If $c_1,\dots,c_m \in \End_\kappa(A) \tensor_{\IZ} \IZ_\ell \subseteq D\tensor_\IQ\IQ_\ell$, are left linearly independent over $D$, then
\begin{equation}
\Gamma:=\{(c_1x,c_2x,\dots,c_mx):x\in A[\ell^\infty]\}\subseteq A^m (\bar{\kappa})
\end{equation}
is Zariski dense in $A^m$.
\end{lemma}

We remark that if $c_1, \cdots, c_m \in \IZ_\ell$, then left linear independence over $D$ is equivalent to linear independence over $\IQ$. 

\begin{proof}
Set $C=\overline{\Gamma}^\Zar$. Then $C$ is an abelian subvariety of $A^m$. Suppose that $C\neq A^m$. By Poincar\'e's reducibility theorem and the simplicity of $A$, there exists a non-trivial homomorphism $\phi:A^m\to A$ such that $C\subseteq\ker(\phi)$. Write
\begin{equation}
\phi(y_1,\dots,y_m)=\sum_{i=1}^m f_i(y_i),
\end{equation}
where $f_i\in\End_\kappa(A)$ are not all zero. Because $C \subseteq\ker(\phi)$ and $\End_\kappa(A)\tensor_{\IZ}\IZ_\ell$ acts faithfully on $A[\ell^\infty]$, we have
\begin{equation*}
\sum_{i=1}^m f_i c_i x=0 \text{ for every $x\in A[\ell^\infty]$} \Rightarrow \sum_{i=1}^m f_i c_i=0. 
\end{equation*}
But the $c_i$ are left linearly independent over $D$, so $f_i=0$ for every $i$, a contradiction.
\end{proof}

\begin{theorem}
\label{thm: Faltings hom}
    Let $A$ and $B$ be abelian varieties over a number field $K$. Then the map 
    $$ \Hom_K(A, B) \tensor \IZ_\ell \to \Hom_{\Gal_K}(T_\ell A, T_\ell B)$$
    is an isomorphism. 
\end{theorem}
\begin{proof}
    By standard reductions, let us assume that $A, B$ are both geometrically simple.
    Choose an unramified place $v$ of $K$ of characteristic $p \neq 2, \ell$ such that $A$ and $B$ have good reduction at $v$. 
    By \cref{thm: Junyi in char p}, the following natural diagram 
\begin{equation*}
    \begin{tikzcd}
	{\Hom_K(A, B)_\IQ} & {\Hom_k(A_0, B_0)_\IQ} \\
	{\Hom_{\Gal_K} (V_\ell A, V_\ell B)} & {\Hom_{\Gal_k}(V_\ell A_0, V_\ell B_0)}
	\arrow[from=1-1, to=1-2]
	\arrow[from=1-1, to=2-1]
	\arrow["\lrcorner"{anchor=center, pos=0.125}, draw=none, from=1-1, to=2-2]
	\arrow[from=1-2, to=2-2]
	\arrow[from=2-1, to=2-2]
\end{tikzcd}
\end{equation*}
is Cartesian. All arrows above are injective, so we view every space above as a subspace of the bottom right corner, i.e., $\Hom_{\Gal_k}(V_\ell A_0, V_\ell B_0)$. Note that by Tate's theorem one may view $\Hom_k (A_0, B_0)_\IQ$ as giving a rational structure to $\Hom_{\Gal_k}(V_\ell A_0, V_\ell B_0)$. From this point of view, Faltings' isogeny theorem is equivalent to saying that $\Hom_{\Gal_K} (V_\ell A, V_\ell B)$ is a \textit{rational} subspace (i.e., admits a rational basis). 

Suppose for the sake of contradiction that $\Hom_{\Gal_K} (V_\ell A, V_\ell B)$ is not rational. Then by \cref{lem: find basis}, we can find a basis $\{ h_1, \cdots, h_r, g_1, \cdots, g_s \}$ for $\Hom_k(A_0, B_0)_\IQ$ such that $\{ h_1, \cdots, h_r \}$ is a basis for $\Hom_K(A, B)_\IQ$, and $\Hom_{\Gal_K} (V_\ell A, V_\ell B)$ contains an element of the form 
$$ \sum_{i = 1}^m c_i g_i, \text{ where $c_1, \cdots, c_m \in \IQ^\times_\ell$ are linearly independent over $\IQ$ for some $m \ge 2$.} $$
We may arrange $c_i\in\IZ_\ell$ and $g_i\in\Hom_k(A_0,B_0)$ by clearing denominators.
Consider 
$$ H := \{ (c_1 x, c_2 x, \cdots, c_m x, \sum_{i = 1}^m g_i(c_i x)) : x \in A[\ell^\infty] \} \subseteq (A^m \times B)(\bar{K}). $$
Take $C = \overline{H}^\Zar \subseteq A^m \times B$, and let $C_0$ be its reduction. Consider the map $G_0 : A_0^m \to B_0$ defined by 
$$ G_0(y_1, y_2, \cdots, y_m) = \sum_{i = 1}^m g_i(y_i). $$
Then the reductions of $H$ are contained in $\Gamma_{G_0}$.
Note that $H$ is $\ell$-divisible and $\Gal_K$-stable. By \cref{prop: tube absorbance}, $C_0 \subseteq \Gamma_{G_0}$. Therefore, $\dim C = \dim C_0 \le m \dim A$. 

By \cref{lem: irrationality}, the projection of $C$ to $A^m$ must be surjective, so that $\dim C \ge m \dim A$. On the other hand, $\dim C = \dim C_0 \le \dim \Gamma_{G_0}$, so we must have $C_0 = \Gamma_{G_0}$. Therefore, the projection $C_0 \to A_0^m$ is an isomorphism. This in turn implies that the projection $C \to A^m$ on the generic fiber is also an isomorphism, which is equivalent to saying that $C \subseteq A^m \times B$ is the graph of some morphism $G:  A^m \to B$. It is not hard to see that $G$ is necessarily of the form $$ G(y_1, y_2, \cdots, y_m) = \sum_{i = 1}^m \wt{g}_i(y_i). $$ for some $\wt{g}_i$ that lifts $g_i$. But by assumption $g_i$'s do not come from $\Hom_K(A, B)_\IQ$. 
\end{proof}

\section{Semisimplicity}
\label{sec: semisimplicity}
Let us introduce some preliminary observations about semisimple algebra. 
\begin{notation}
Let $V$ be a vector space over a field $k$ and $D \subseteq \End_k(V)$ be a $k$-subalgebra. We shall use the following notation.
\begin{itemize}
    \item $D' := \End_D(V)$ denotes the centralizer of $D$ in $\End_k(V)$.
    \item For a subspace $W \subseteq V$, set $L_D(W) := \{d \in D : dW=0\}$ and call it the left annihilator of $W$; write 
    $$ L^\perp_D(W) := \{e \in L_D(W) : e^2=e,\ \ker(e)=W\} \subseteq L_D(W) $$
    for the subset of idempotent projectors which kill exactly $W$. 
    \item For a subalgebra $C \subseteq D$, subset $S \subseteq D$ and a subspace $W \subseteq V$,
    $$ CW := \sum_{f\in C} f(W), \text{ and } \ker(S) := \bigcap_{f\in S}\ker(f). $$
    Note that $CW$ is the smallest $C$-invariant subspace of $V$ containing $W$.
\end{itemize}
\end{notation}

Let us recall a basic fact about semisimple algebras. 
\begin{lemma}
\label{prop: gen'd by idempotents}
    Let $D$ be a semisimple algebra over a field $k$. Then every left ideal $I$ is of the form $De$ for some idempotent $e$. If $D \subseteq \End_k(V)$ for a vector space $V$ over $k$, then $I = De$ for any idempotent $e \in I$ with $\ker(e) = \ker(I)$; in fact, for every $i \in I$, $i = ie$. 
\end{lemma}
\begin{proof}
    The short exact sequence $0 \to I \to D \to D/I \to 0$ of left $D$-modules splits, so we have a projection $\pi : D \to I$ of $D$-modules. Set $e = \pi (1)$. Then for every $d \in D$, $\pi(d) = \pi(d \cdot 1) = d \pi(1) = de$. 
    
    We now prove the second assertion. Let $e \in I$ be any idempotent such that $\ker(e) = \ker(I)$. Then $De \subseteq I$ by default. Now, every $v \in V$ uniquely decomposes to $v_0 + v_1$ where $v_0 \in \ker(e)$ and $v_1 \in \mathrm{im}(e)$. For every $i \in I$, we have $i(v_0 + v_1) = i(v_1) = i e (v_1) + ie(v_0)$ as $i$ also kills $v_0$. Hence $i = i e \in De$. 
\end{proof}

\begin{lemma}
\label{lem: produce idempotent}
Let $V$ be a finite dimensional vector space over a field $k$.
Let $D \subseteq \End_k(V)$ be a semisimple $k$-algebra and $W \subseteq V$ a subspace. If $e \in D$ is an idempotent such that $L_D(W)=De$, then $\ker(L_D(W))=\ker(e)=D'W$. 
\end{lemma}

\begin{proof}
Since $L_D(W)=De$ and $e\in De$, we have
\(\ker(L_D(W))=\ker(De)=\ker(e).\)
Moreover, $D'W\subseteq \ker(e)$ because $eW=0$ and $e$ commutes with $D'$.

For the reverse inclusion, since $D$ is semisimple, so is $D'$, and there is a $D'$-linear projection $\pi:V\to D'W$. 
By the double centralizer theorem, $\pi\in D''=D$. Since $\pi$ restricts to the identity on $D'W$, and hence on $W$, we have $(1-\pi)W=0$, so $1-\pi\in L_D(W)$. Thus, if $v\in\ker(L_D(W))$, then $(1-\pi)v=0$, and therefore $v=\pi(v)\in D'W$. This implies $\ker(L_D(W))\subseteq D'W$.
\end{proof}

\begin{lemma}
\label{lem: relative position}
Let $V$ be a finite dimensional vector space over a field $k$ and $W \subseteq V$ be a subspace. 
Suppose that $D \subseteq B \subseteq \End_k(V)$ are semisimple $k$-algebras and $L_B^\perp(W)\neq \emptyset$. Then 
\(L_B(W)=B\cdot L_D(W)\)
if and only if $L_D^\perp(W)\neq \emptyset$.
\end{lemma}

\begin{proof}
Suppose first that $L_B(W)=B\cdot L_D(W)$, and choose $e\in L_B^\perp(W)$. Since $e\in B\cdot L_D(W)$, we have
$\ker(L_D(W)) \subseteq\ker(e)=W.$
The reverse inclusion is tautological, so $\ker(L_D(W))=W$. By \cref{prop: gen'd by idempotents}, there is an idempotent $d\in L_D(W)$ with $L_D(W)=Dd$. 
Again by \cref{prop: gen'd by idempotents}, $\ker(d)=\ker(L_D(W))=W$, so $d\in L_D^\perp(W)$.

Conversely, suppose that $d\in L_D^\perp(W)$. Since $d\in L_D(W)\subseteq L_B(W)$ and $\ker(d)=W=\ker(L_B(W))$, 
\cref{prop: gen'd by idempotents} gives $L_D(W)=Dd$ and $L_B(W) = Bd$. Therefore
$B\cdot L_D(W) = BDd = Bd = L_B(W)$. 
\end{proof}

\begin{corollary}
\label{cor: relative position}
    Suppose that $A$ is an abelian variety over a number field $K$, $v$ is a place of good reduction with residue field $k$, and $\ell\neq\mathrm{char\,}k$. Let $A_0$ be the special fiber and identify $V_\ell:=V_\ell A=V_\ell A_0$. Set $D=\End_K(A)_\IQ$, $B=\End_k(A_0)_\IQ$, $D_\ell  = D \tensor_\IQ \IQ_\ell$ and $B_\ell = B \tensor_\IQ \IQ_\ell$. Then the action of $\Gal_K$ on $V_\ell$ is semisimple if and only if
    \begin{equation*}
        L_{B_\ell}(W_\ell)=B_\ell\cdot L_{D_\ell}(W_\ell)
    \end{equation*}
    for every $\Gal_K$-invariant subspace $W_\ell\subseteq V_\ell$.
\end{corollary}

\begin{proof}
    Since $D$ and $B$ are semisimple, so are their scalar extensions $D_\ell$ and $B_\ell$ (see e.g., \cite[\S7.9]{Voi21}).
    By \cref{thm: Faltings hom}, $D_\ell=\End_{\Gal_K}(V_\ell)$. Hence the action of $\Gal_K$ on $V_\ell$ is semisimple if and only if $L_{D_\ell}^\perp(W_\ell)\neq \emptyset$ for every $\Gal_K$-invariant subspace $W_\ell$. Now choose such a $W_\ell$. As the $\Gal_k$-action on $V_\ell$ is semisimple, $L_{B_\ell}^\perp(W_\ell) \neq \emptyset$. The result now follows from \cref{lem: relative position}.
\end{proof}

\begin{remark}
\label{rmk: pump semisimplicity}
In \cref{lem: relative position}, the hypothesis $L_B^\perp(W)\neq \emptyset$ is certainly necessary, since $L_D^\perp(W)\subseteq L_B^\perp(W)$. While this observation is trivial, it is precisely at this point in \cref{cor: relative position} that we use the semisimplicity of the $\Gal_k$-action on $V_\ell A_0$, and this is why we say the equidistribution argument to be given below ``pumps'' semisimplicity from finite fields to number fields.
\end{remark}

Much of linear algebra works when the scalar field is replaced by a division algebra. For example, the ``minimal support trick'' used in the proof of \cref{lem: find basis} works for possibly non-commutative division algebras as well. We shall need the following. 
\begin{lemma}
\label{lem: minimal tensor expression}
Let $D$ be a division algebra, let $M$ be a right $D$-module, and let $N$ be a left $D$-module. Suppose that $0\neq q\in M\tensor_D N$ and that $$q=\sum_{i=1}^m x_i\tensor y_i$$ is an expression with $m$ minimal. Then $x_1,\dots,x_m$ are right linearly independent over $D$, and $y_1,\dots,y_m$ are left linearly independent over $D$.
\end{lemma}

\begin{proof}
Suppose first that $x_1,\dots,x_m$ are right linearly dependent. After relabelling, there is a relation $$\sum_{i=1}^m x_i d_i=0 \text{ with } d_1\neq 0. $$ Since $D$ is a division algebra, $d_1$ is invertible, so $$x_1=-\sum_{i=2}^m x_i d_i d_1^{-1}, \text{ and hence } q=\sum_{i=2}^m x_i\tensor (y_i-d_i d_1^{-1}y_1). $$ 
This contradicts the minimality of $m$. The argument for the left $D$-linear independence of $y_i$'s is symmetric. 
\end{proof}

Finally, we are ready to prove the semisimplicity theorem. 

\begin{theorem}\label{thm:main-semisimplicity}
    Let $A$ be an abelian variety over a number field $K$. Then the $\Gal_K$-action on $V_\ell A$ is semisimple. 
\end{theorem}

\begin{proof}
    Again by standard reductions, it suffices to consider the case when $A$ is simple. 
    We keep using the set-up in \cref{cor: relative position}, except that below we simply write $V, W$ for $V_\ell, W_\ell$. Note that the simplicity of $A$ implies that $D$ is a division algebra.  Our goal is to show that for every $\Gal_K$-invariant subspace $W \subseteq V$, 
    \begin{equation*}
        L_{B_\ell} (W) = B_\ell \cdot L_{D_\ell}(W). 
    \end{equation*}
    That is, the annhilator of $W$ in $B_\ell$ is generated as a left $B_\ell$-module by the annhilator of $W$ in $D_\ell$. 
    
    Let us first note that as $B_\ell = B \tensor_D D_\ell$, there is an identification 
    \begin{equation*}
        B_\ell / [B_\ell \cdot L_{D_\ell}(W)] = B \tensor_D [D_\ell / L_{D_\ell}(W)].
    \end{equation*}
    Here $B$ is regarded as a right $D$-module and $D_\ell/L_{D_\ell}(W)$ as a left $D$-module.

    Clearly, $B_\ell \cdot L_{D_\ell}(W) \subseteq L_{B_\ell}(W)$. Suppose now that it is not an equality. Then we can choose some $q \in L_{B_\ell}(W) \smallsetminus B_\ell \cdot L_{D_\ell}(W)$. Write $\bar{q}$ for the image of $q$ in the quotient $B_\ell / [B_\ell \cdot L_{D_\ell}(W)]$, and find an expression
    $$ \bar{q} = \sum_{i = 1}^m g_i \tensor \bar{c}_i \text{ for } g_i \in B \text{ and } \bar{c}_i \in D_\ell/L_{D_\ell}(W) $$
    that minimizes $m$. Then by \cref{lem: minimal tensor expression}, $g_i$'s are right $D$-linearly independent and $\bar{c}_i$'s are left $D$-linearly independent. Choose a lift $c_i \in D_\ell$ for each $\bar{c}_i$. Then we obtain a lift $$q' = \sum_{i = 1}^m g_i c_i$$ for $\bar{q}$, so that $\delta := q - q' \in B \tensor_D L_{D_\ell}(W) = B_\ell \cdot L_{D_\ell}(W)$. Clearly, as $q$ and $\delta$ both kill $W$, so does $q'$. 

    By scaling the relation if needed, we may arrange $g_i\in\End_k(A_0)$ to be actual endomorphisms and $c_i\in\End_K(A)\tensor\IZ_\ell$, so they act on $A[\ell^\infty]$. Define $H_W = W / (W \cap T_\ell A_0)$ to be the $\ell$-divisible group given by $W$. Note that it is $\Gal_K$-invariant. Now we apply the diagonal trick again and consider 
    $$ H = \{ (c_1 x, \cdots, c_m x) : x \in H_W \} \subseteq A^m [\ell^\infty]. $$

    Here we need a little adaptation of \cref{lem: irrationality}. By construction $c_i$'s remain left $D$-linearly independent modulo $L_{D_\ell}(W)$. If $\overline{H}^\Zar \neq A^m$, there exist morphisms $d_i : A \to A$, not all zero, such that 
    $$ \overline{H}^\Zar \subseteq \ker \big( \sum_{i = 1}^m d_i : A^m \to A \big) \Rightarrow (\sum_{i = 1}^m d_i c_i) W=0 \Rightarrow \sum_{i = 1}^m d_i \bar{c}_i = 0, $$
    which is a contradiction. Hence $\overline{H}^\Zar = A^m$. On the other hand, by \cref{prop: tube absorbance}, the reduction of $\overline{H}^\Zar$ has to lie in the kernel of the map 
    $$ \sum_{i = 1}^m g_i : A_0^m \to A_0,  $$
    as $H$ does. But the $g_i$'s are right $D$-linearly independent, the above map is nonzero, which is absurd.  
\end{proof}

\vspace{1em}

\paragraph{Acknowledgments} This project began during the ``Fifth National Conference on Algebraic Geometry'' held at Peking University, where the second author was an invited speaker and the first author explained his toy example (\cref{thm: Junyi over C}) to the second author. We thank Peking University for its hospitality. We also thank Ziyang Gao, who joined our initial discussion, and Shouwu Zhang for asking whether Faltings' isogeny theorem could be proved via equidistribution. We are also grateful to Xinyi Yuan for many valuable discussions and for his help in checking our proofs. 

The first author is supported by NSFC Grant
(No.12271007) and by the Xplorer
Prize from the New Cornerstone Science Foundation.
The second author is supported by CUHK start up grant, and an ECS grant from the Hong Kong RGC (Project No. 24308225). 

\medskip

\paragraph{AI Usage} The authors used ChatGPT 5.6 Sol to assist with semisimple algebra. Specifically, we prompted AI with our proof of the Hom theorem, and asked it to search for a relation between endomorphisms algebras so that the diagonal trick still applies. The idea of considering the relation in \cref{cor: relative position} is suggested by AI. The manuscript is written by the authors, who are responsible for its correctness. 

\medskip

\paragraph{Afterword} 
After the main text was posted to arXiv, Jit Wu Yap and Matt H. Baker independently used GPT-6 Astra to explore whether our method could be extended to prove \cref{thm:app-isogeny-finiteness} below and sent us an initial draft. We thank them for their interest in our work. We decided to formalize and include the argument as an appendix, as it helps complete the narrative leading to Mordell. We verified the proof with Yap and thoroughly revised the exposition to match the style of the main text.

\appendix
\section{Finiteness in an isogeny class}
\label{sec:appendix}
\begin{center}
\textit{Joint with Jit Wu Yap}
\end{center}

This appendix extends the method of the main text to the following classical finiteness theorem of Faltings \cite{Fal83}. Given \cref{cor: isogenous}, standard arguments with Galois representations show that, for any fixed finite set $S$ of finite places of $K$, there are only finitely many $K$-isogeny classes of abelian varieties over $K$ of a fixed dimension and with good reduction outside $S$. Combining this with \cref{thm:app-isogeny-finiteness} solves the Shafarevich problem for abelian varieties, and hence, via the Shafarevich–Parshin reduction, the Mordell conjecture.
Therefore, the method of the main text gives yet another route to the Mordell conjecture, after the alternative proofs of Vojta \cite{Voj91}, Masser--W{\"u}stholz \cite{MW93}, and Lawrence--Venkatesh \cite{LV20}.

\begin{theorem}\label{thm:app-isogeny-finiteness}
Let $K$ be a number field and let $A$ be an abelian variety over $K$. There are only finitely many $K$-isomorphism classes of abelian varieties over $K$ that are $K$-isogenous to $A$.
\end{theorem}

\subsection{Varying coefficients}\label{sec:01-residual} In this subsection, we use ultraproducts to prove the residual Hom theorem and semisimplicity for all sufficiently large primes.

\medskip

\paragraph{Integral actions and ultraproducts.}  For background on ultrafilters and ultraproducts, see \cite[Chapter~2]{Sch10}. Let $\Lambda$ be an infinite set of primes and fix a nonprincipal ultrafilter $\sU$ on $\Lambda$; call its members \emph{large}. Removing finitely many elements from a large set preserves largeness, and every finite partition of $\Lambda$ has exactly one large part. The ultraproduct
\[
F:=\prod_{\sU}\IF_\ell=\Big(\prod_{\ell\in\Lambda}\IF_\ell\Big)/\sim_{\sU}
\]
is a field: a nonzero class can be inverted on a large set. Here ``$\sim_\sU$'' means that two sequences are equivalent if and only if they agree on a large set.  Since every nonzero integer has only finitely many prime divisors, $F$ has characteristic zero and contains $\IQ$. We write its elements as $\{a_\ell\}_\sU$.

For every finite free $\IZ$-module $M$, there is a canonical identification
\begin{equation}\label{eq:residual-ultraproduct-coordinates}
M\tensor_\IZ F\xrightarrow{\sim}\prod_\sU(M/\ell M),\qquad m\tensor\{a_\ell\}_\sU\longmapsto\{m\tensor a_\ell\}_\sU.
\end{equation}
If $M$ is a ring, this identification preserves multiplication, since its multiplication table has fixed integral structure constants.

A collection $\{V_\ell\}_{\ell\in\Lambda}$ of $\IF_\ell$-vector spaces of uniformly bounded dimension has an ultraproduct $V:=\prod_\sU V_\ell$, whose dimension is the unique $r$ such that $\{\ell:\dim V_\ell=r\}$ is large. We freely use {\L}o\'s theorem in the form \cite[Thms.~2.3.1--2.3.2]{Sch10}: after choosing bases, finite systems of polynomial equations and inequations in matrix entries may be checked in the ultraproduct, or equivalently on a large set. In particular, this applies to polynomial identities, ranks, and kernels of maps between vector spaces of fixed dimension.

\paragraph{The varying-prime tube argument.} The proof of \cref{prop: tube absorbance} also gives the following version for torsion of varying prime-to-$p$ order.

\begin{proposition}\label{prop:app-tube}
Let $X,K,v,k$ be as in \cref{prop: tube absorbance}, with $p:=\operatorname{char}k$, and let $H\subseteq X(\bar K)$ be a $\Gal_K$-invariant subgroup of prime-to-$p$ torsion points. Set $C:=\overline H^\Zar$. If $Z_0\subseteq X_0$ is closed and $H\subseteq]Z_0[$, then $C^0\subseteq]Z_0[$. Moreover, for every prime $\ell\nmid|\pi_0(C_{\bar K})|$, one has $H\cap X[\ell]\subseteq C^0$.
\end{proposition}

\begin{proof}
The subgroup $H\cap C^0$ is Zariski dense in $C^0$, so the proof of \cref{prop: tube absorbance} applies verbatim to a generic torsion sequence in $H\cap C^0$. For the last assertion, the image of $H\cap X[\ell]$ in $\pi_0(C_{\bar K})$ is $\ell$-torsion, hence is trivial when $\ell\nmid|\pi_0(C_{\bar K})|$.
\end{proof}

\paragraph{The residual Hom theorem.} The following is the residual analogue of \cref{thm: Faltings hom}.

\begin{theorem}\label{thm:app-residual-hom}
Let $A,B$ be abelian varieties over a number field $K$. For all but finitely many primes $\ell$, the natural map $\Hom_K(A,B)\tensor\IF_\ell\to\Hom_{\Gal_K}(A[\ell],B[\ell])$ is an isomorphism.
\end{theorem}

\begin{proof}
We follow the proof of \cref{thm: Faltings hom}. By standard reductions, we assume that $A$ is $K$-simple. Choose a common good-reduction place $v$ above an odd prime $p$ unramified in $K$. Write $k$ for the residue field, $A_0, B_0$ for the special fibers, $M:=\Hom_K(A,B)$, and $L:=\Hom_k(A_0,B_0)$.  View $M\subseteq L$ by specialization. Note that the torsion subgroup of $L/M$ is $p$-primary (e.g., by \cite[Thm.~3.15]{BO83}). Hence, for every $\ell\neq p$, $M/\ell M$ injects to $L/\ell L$, and we view it as subspace. 

Suppose the residual Hom map fails to be surjective for an infinite set $\Lambda$. After removing $p$ and finitely many further primes from $\Lambda$, we assume that for every $\ell \in \Lambda$, $A_0[\ell]$ is a semisimple $\Gal_k$-module and $\Hom(A_0, B_0) \to \Hom_{\Gal_k}(A_0[\ell], B_0[\ell])$ is surjective (\cite{Zar77}). For each $\ell\in\Lambda$, choose $\psi_\ell\in\Hom_{\Gal_K}(A[\ell],B[\ell])$ outside $M/\ell M$. By assumption, $\psi_\ell=b_\ell|_{A[\ell]}$ for some $b_\ell\in L/\ell L$. Choose a nonprincipal ultrafilter $\sU$ on $\Lambda$, put $F:=\prod_\sU\IF_\ell$, and let $b:=\{b_\ell\}_\sU\in L\tensor_\IZ F$.  Since $b_\ell\notin M/\ell M$ for every $\ell\in\Lambda$, it follows that
$b\notin M\tensor_\IZ F$.

After subtracting an element of $M\tensor_\IZ F$ from $b$, and replacing $\psi_\ell$ and $b_\ell$ by the corresponding adjusted classes, the minimal-support argument of \cref{lem: find basis} gives
\[
b=\sum_{i=1}^m g_i\tensor c_i, \text{ for some } g_i\in L \smallsetminus M_\IQ,\quad c_i\in F,
\]
with the $c_i$ linearly independent over $\IQ$. Write $c_i=\{c_{i,\ell}\}_\sU$. On a large subset $\Lambda_0\subseteq\Lambda$,
\begin{equation}\label{eq:residual-hom-identity}
b_\ell=\sum_i c_{i,\ell}g_i\quad\text{in }L/\ell L.
\end{equation}

For $\ell\in\Lambda_0$, set $H_\ell:=\{(c_{1,\ell}x,\ldots,c_{m,\ell}x,\psi_\ell(x)):x\in A[\ell]\}\subseteq(A^m\times B)[\ell]$ and let $H$ be the subgroup of generated by these $H_\ell$'s for $\ell \in \Lambda_0$. These groups are $\Gal_K$-stable and, by \eqref{eq:residual-hom-identity}, reduce into the graph of
\[
G_0:A_0^m\to B_0,\qquad G_0(y_1,\ldots,y_m):=\sum_i g_i(y_i).
\]
Let $C:=(\overline H^\Zar)^0$. By \cref{prop:app-tube} and \cite[\S7.5, Thm.~4]{BLR90}, the reduction $C_0\to A_0^m\times B_0$ is a closed immersion with image contained in $\Gamma_{G_0}$; moreover, $H_\ell\subseteq C$ for all $\ell$ in a large subset of $\Lambda_0$.

If the projection $C\to A^m$ were not surjective, the $K$-simplicity of $A$ would give $d_1,\ldots,d_m\in\End_K(A)$, not all zero, such that $\sum_i d_i c_{i,\ell}=0$ in $\End_K(A) \tensor \IF_\ell $ on a large set. Taking ultraproducts gives $\sum_i d_i\tensor c_i=0$, contradicting the $\IQ$-linear independence of the $c_i$. Thus $C\to A^m$ is surjective. Since $C_0\subseteq\Gamma_{G_0}$, dimension forces $C_0=\Gamma_{G_0}$, and hence $C$ is the graph of a $K$-homomorphism $A^m\to B$ lifting $G_0$. Restricting to the factors of $A^m$ shows that every $g_i$ lies in $M$, a contradiction.
\end{proof}

\paragraph{Residual semisimplicity.} The following is the residual analogue of \cref{thm:main-semisimplicity}.

\begin{theorem}\label{thm:app-residual-semisimplicity}
Let $A$ be an abelian variety over a number field $K$. For all but finitely many primes $\ell$, the $\IF_\ell[\Gal_K]$-module $A[\ell]$ is semisimple.
\end{theorem}

\begin{proof}
We follow the proof of \cref{thm:main-semisimplicity}. By standard reductions, assume that $A$ is $K$-simple. Choose a good-reduction place $v$ above an odd prime $p$ unramified in $K$, with residue field $k$, and write $A_0$ for the special fiber. Put $\sO:=\End_K(A)$, $D:=\sO\tensor_\IZ\IQ$, and $E:=\End_k(A_0)\tensor_\IZ\IQ$. Then $D\subseteq E$ and $D$ is a division algebra.

Suppose that $A[\ell]$ is not semisimple for an infinite set $\Lambda$ of $\ell$'s. Again by \cite{Zar77}, after discarding $p$ and finitely many primes, the $\Gal_k$-module $A_0[\ell]$ is semisimple and $\End_k(A_0)/\ell\End_k(A_0)\to\End_{\Gal_k}(A_0[\ell])$ is an isomorphism for every $\ell \in \Lambda$. For each remaining $\ell \in \Lambda$, choose a $\Gal_K$-stable subspace $W_\ell\subseteq A[\ell]$ with no stable complement. Fix a nonprincipal ultrafilter $\sU$ on $\Lambda$ and put
\[
F:=\prod_\sU\IF_\ell,\qquad V:=\prod_\sU A[\ell],\qquad W:=\prod_\sU W_\ell,\qquad D_F:=D\tensor_\IQ F,\qquad E_F:=E\tensor_\IQ F.
\]

Suppose that $L_{E_F}(W)\neq E_F L_{D_F}(W)$. Repeating the minimal-tensor-expression argument in the proof of \cref{thm:main-semisimplicity}, we obtain right $D$-linearly independent $g_1,\ldots,g_m\in\End_k(A_0)$ and $c_1,\ldots,c_m\in D_F$ whose classes in $D_F/L_{D_F}(W)$ are left $D$-linearly independent, such that
\[
\Big(\sum_i g_i c_i\Big)W=0.
\]
Represent $c_i$ by $c_{i,\ell}\in\sO/\ell\sO$ using \eqref{eq:residual-ultraproduct-coordinates}. On a large subset $\Lambda_0$, one has $(\sum_i g_i c_{i,\ell})W_\ell=0$. For $\ell\in\Lambda_0$, set
\[
H_\ell:=\{(c_{1,\ell}x,\ldots,c_{m,\ell}x):x\in W_\ell\},\qquad H:=\langle H_\ell:\ell\in\Lambda_0\rangle.
\]
The $H_\ell$ are $\Gal_K$-stable and reduce into the kernel of the fixed nonzero homomorphism
\[
G_0:A_0^m\to A_0,\qquad G_0(y_1,\ldots,y_m):=\sum_i g_i(y_i).
\]
Let $C:=(\overline H^\Zar)^0$. By \cref{prop:app-tube}, $H_\ell\subseteq C$ for all $\ell$ in a large subset of $\Lambda_0$, and $C_0\subseteq\ker(G_0)$. If $C\neq A^m$, the $K$-simplicity of $A$ gives fixed $d_1,\ldots,d_m\in\End_K(A)$, not all zero, such that $(\sum_i d_i c_{i,\ell})W_\ell=0$ on a large set. Taking ultraproducts gives $\sum_i d_i c_i\in L_{D_F}(W)$, contradicting the left $D$-linear independence of the classes of the $c_i$. Thus $C=A^m$, contradicting $C_0\subseteq\ker(G_0)$ and $G_0\neq0$. 

Now we know $L_{E_F}(W)=E_F L_{D_F}(W)$. For each $\ell \in \Lambda$, semisimplicity of $A_0[\ell]$ gives a $\Gal_k$-equivariant projector with kernel $W_\ell$, and Zarhin's theorem identifies it with an idempotent $a_\ell\in\End_k(A_0)/\ell\End_k(A_0)$. Their ultraproduct is an idempotent $a\in E_F$ with kernel $W$. Since $D_F$ and $E_F$ are semisimple, \cref{lem: relative position} gives an idempotent $d\in D_F$ with $\ker d=W$. Representing $d$ by $d_\ell\in\sO/\ell\sO$, {\L}o\'s's theorem gives $d_\ell^2=d_\ell$ and $\ker d_\ell=W_\ell$ on a large set. Then $\mathrm{im}(d_\ell)$ is a $\Gal_K$-stable complement to $W_\ell$, a contradiction.
\end{proof}

\subsection{From stable lattices to isomorphism classes}\label{sec:02-isogeny}

We now prove \cref{thm:app-isogeny-finiteness}. We first establish a finiteness result for Galois-stable lattices. For a finite-dimensional $\IQ_\ell$-vector space $V_\ell$, a lattice means a finite free $\IZ_\ell$-submodule $L_\ell\subseteq V_\ell$ spanning $V_\ell$ over $\IQ_\ell$.

\begin{proposition}\label{thm:app-lattice-finiteness}
Let $\Gamma$ be a group and $\sO$ a unital, possibly noncommutative finite free $\IZ$-algebra. For every prime $\ell$, let $T_\ell$ be a finite free $\IZ_\ell$-module with commuting $\IZ_\ell$-linear actions of $\Gamma$ and $\sO$, and put $V_\ell:=T_\ell[1/\ell]$. Let $\mathscr X$ be the set of families $L=(L_\ell)_\ell$ such that each $L_\ell\subseteq V_\ell$ is a $\Gamma$-stable lattice and $L_\ell=T_\ell$ for all but finitely many $\ell$. Assume:
\begin{enumerate}[(i)]
\item for every $\ell$, the $\Gamma$-module $V_\ell$ is semisimple and $\sO\tensor_\IZ\IZ_\ell\to\End_\Gamma(T_\ell)$ is an isomorphism;
\item for all but finitely many $\ell$, the $\Gamma$-module $T_\ell/\ell T_\ell$ is semisimple and $\sO\tensor_\IZ\IF_\ell\to\End_\Gamma(T_\ell/\ell T_\ell)$ is an isomorphism.
\end{enumerate}
Then $D^\times\backslash\mathscr X$ is finite, where $D:=\sO\tensor_\IZ\IQ$.
\end{proposition}

\begin{proof}
By (i), $D$ is semisimple and $\End_\Gamma(V_\ell)=D\tensor_\IQ\IQ_\ell$. Hence \cref{lem:app-local-lattices} shows that the $\Gamma$-stable lattices in $V_\ell$ have finitely many $(D\tensor_\IQ\IQ_\ell)^\times$-orbits for every $\ell$, while (i)--(ii) give exactly one orbit for all but finitely many $\ell$. Thus $D_f:=D\tensor_\IQ\IA_f$ has finitely many $D_f^\times$-orbits on $\mathscr X$.

Fix a representative $M=(M_\ell)_\ell$ of such an orbit and let $U_M\subseteq D_f^\times$ be its stabilizer, which is compact open. The $D^\times$-orbits in $D_f^\times M$ are parametrized by $D^\times\backslash D_f^\times/U_M$, which is finite by Borel's class-number theorem \cite[Theorem~5.1]{Bor63}. Since there are only finitely many $D_f^\times$-orbits on $\mathscr X$, the result follows.
\end{proof}

\begin{lemma}\label{lem:app-local-lattices}
Let $\Gamma$ act $\IZ_\ell$-linearly on a nonzero finite free $\IZ_\ell$-module $T_\ell$, and put $V_\ell:=T_\ell\tensor_{\IZ_\ell}\IQ_\ell$ and $D_\ell:=\End_\Gamma(V_\ell)$. If $V_\ell$ is semisimple as a $\Gamma$-module, then the $\Gamma$-stable lattices in $V_\ell$ have finitely many $D_\ell^\times$-orbits. If, in addition, $T_\ell/\ell T_\ell$ is semisimple and $\End_\Gamma(T_\ell)\to\End_\Gamma(T_\ell/\ell T_\ell)$ is surjective, there is exactly one such orbit.
\end{lemma}

\begin{proof}
Let $R_\ell$ be the $\IZ_\ell$-span of the image of $\Gamma$ in $\End_{\IZ_\ell}(T_\ell)$ and set $E_\ell:=R_\ell\tensor_{\IZ_\ell}\IQ_\ell$. Then $\End_{E_\ell}(V_\ell)=D_\ell$, and $E_\ell$ is semisimple because $V_\ell$ is a faithful semisimple $E_\ell$-module. Choose a maximal $\IZ_\ell$-order $S_\ell\subseteq E_\ell$ containing $R_\ell$ \cite[(10.4)]{Rei03}, and choose $a\geq0$ with $\ell^aS_\ell\subseteq R_\ell$.

By \cite[Theorem~18.10]{Rei03}, any two $S_\ell$-stable lattices in $V_\ell$ are isomorphic as $S_\ell$-modules; extending scalars gives an element of $D_\ell^\times$. Thus the $S_\ell$-stable lattices form one $D_\ell^\times$-orbit. Every $\Gamma$-stable lattice $L_\ell$ satisfies $\ell^aS_\ell L_\ell\subseteq L_\ell\subseteq S_\ell L_\ell$. After moving $S_\ell L_\ell$ to $S_\ell T_\ell$ by an element of $D_\ell^\times$, we have $\ell^aS_\ell T_\ell\subseteq L_\ell\subseteq S_\ell T_\ell$. Since there are only finitely many such intermediate lattices, the first assertion holds. 

For the second, let $L_\ell$ be a $\Gamma$-stable lattice and scale it by a power of $\ell$ so that $L_\ell\subseteq T_\ell$. We induct on $[T_\ell:L_\ell]$. If $L_\ell\neq T_\ell$, then $L_\ell+\ell T_\ell\neq T_\ell$ by Nakayama. Semisimplicity gives a $\Gamma$-equivariant projector $\bar e$ on $T_\ell/\ell T_\ell$ with image $(L_\ell+\ell T_\ell)/\ell T_\ell$. By surjectivity and idempotent lifting, $\bar e$ lifts to an idempotent $e\in\End_\Gamma(T_\ell)$. Set $d:=e+\ell(1-e)\in D_\ell^\times$. Then $dT_\ell=L_\ell+\ell T_\ell$, so $d^{-1}L_\ell\subseteq T_\ell$ is $\Gamma$-stable and has strictly smaller index. Now we conclude by induction.
\end{proof}

\begin{proof}[Proof of \cref{thm:app-isogeny-finiteness}]
The case $\dim A=0$ is immediate, so assume $\dim A>0$. Apply \cref{thm:app-lattice-finiteness} with $\Gamma:=\Gal_K$, $\sO:=\End_K(A)$, and $T_\ell:=T_\ell A$. Conditions (i) follow from \cref{thm: Faltings hom,thm:main-semisimplicity}, and (ii) from \cref{thm:app-residual-hom,thm:app-residual-semisimplicity}. Thus, for $D:=\End_K(A)\tensor_\IZ\IQ$, the set $D^\times\backslash\mathscr X$ is finite.

Let $\mathcal I(A)$ be the set of $K$-isomorphism classes of abelian varieties $K$-isogenous to $A$. For $B\in\mathcal I(A)$, choose a $K$-isogeny $f_B:B\to A$ and set $L(B)_\ell:=f_B(T_\ell B)$. Then $L(B):=(L(B)_\ell)_\ell\in\mathscr X$, and the class of $L(B)$ in $D^\times\backslash\mathscr X$ is independent of all choices. Hence we obtain a map $\Phi:\mathcal I(A)\to D^\times\backslash\mathscr X$.

It remains to show that $\Phi$ is injective, so that we are done because $D^\times\backslash\mathscr X$ is finite. If $\Phi(B)=\Phi(C)$, choose $d\in D^\times$ with $dL(B)=L(C)$ and set $u:=f_C^{-1}df_B\in\Hom_K(B,C)\tensor_\IZ\IQ$. Then $u(T_\ell B)=T_\ell C$ for every $\ell$. This immediately implies that $u$ is a $K$-homomorphism. The same argument applied to $u^{-1}$ shows that $u$ is a $K$-isomorphism. 
\end{proof}

\printbibliography

\end{document}